\RequirePackage{fix-cm}
\documentclass[smallextended]{svjour3}       % onecolumn (second format)
\smartqed  % flush right qed marks, e.g. at end of proof
\usepackage{graphicx}
\usepackage{amsmath}
\usepackage[misc]{ifsym}
\usepackage{mathptmx}      % use Times fonts if available on your TeX system
\newcommand{\bc} {\begin{corollary}}
\newcommand{\ec} {\end{corollary}}
\newcommand{\bd} {\begin{definition}}
\newcommand{\ed} {\end{definition}}
\newcommand{\bl} {\begin{lemma}}

\newcommand{\bt} {\begin{theorem}}
\newcommand{\et} {\end{theorem}}
\newcommand{\be} {\begin{equation}}
\newcommand{\ee} {\end{equation}}
\newcommand{\br} {\begin{remark}}
\newcommand{\er} {\end{remark}}
\begin{document}

\title{Global Dynamics of a  Predator-Prey Model with State-Dependent Maturation-Delay}
\author{Qianqian Zhang  \and Yuan Yuan
       \and Yunfei Lv \and Shengqiang Liu\footnote{Corresponding author. \email{sqliu@tiangong.edu.cn}}  %etc.
}

\authorrunning{Zhang, Yuan, Lv and Liu} % if too long for running head

\institute{Qianqian Zhang\at School of Mathematics, Harbin Institute of Technology, Harbin
150001, China  
           \and
           Yuan Yuan \at
            Department of Mathematics and Statistics, Memorial University of Newfoundland, St. John's,
NL, A1C 5S7, Canada \and Yunfei Lv, Shengqiang Liu \at School of Mathematical Sciences, Tiangong University, Tianjin 300387, China
    }

\date{ }
% The correct dates will be entered by the editor

\maketitle

\begin{abstract}
In this paper, a stage structured predator-prey model with general nonlinear type of functional response  is established and analyzed. The state-dependent time delay (hereafter   SDTD) is the time taken from predator's birth to its maturity, formatted  as
a monotonical (ly) increasing, continuous(ly) differentiable and bounded   function  on the number of mature predator. The model is quite different from many previous models with SDTD, in the sense that the  derivative of delay on the time is involved in the model. First, we have shown that
for a large class of commonly used types of functional responses, including Holling types I, II  and III, Beddington-DeAngelis-type (hereafter   BD-type), etc,
the predator
coexists with the prey permanently if and only if the predator's net reproduction number is larger than one unit;
Secondly, we have discussed the local stability of the equilibria of the model;
Finally, for the special case of BD-type functional response, we claim that if the system is permanent, that is, the derivative of SDTD on the state is small enough  and the predator interference is large enough, then  the coexistence equilibrium is globally asymptotically stable.
\end{abstract}

% \PACS{PACS code1 \and PACS code2 \and more}
% \subclass{MSC code1 \and MSC code2 \and more}

\section{Introduction}
\label{intro}
In the modeling of natural ecosystems, after Hutchison proposed a Logistic population model with delay, the time delay differential equation models have received widely attention \cite{1,2}. In addition, since the growth process of mammalian has gone through immature and mature stages
there is  mature time delay, and their behaviors are different at different stages. Therefore, it is necessary to consider the stage structure in the population model. Gurney et al. \cite{3,4} established a stage structure model of the green-headed fly with mature time delay, and numerical simulation based on the experiment data  of the green-headed fly experiment of Nicholson   verified the rationality of the model in the biological sense.

%(This means that in the study of the population model, it is more realistic to consider the time delay and state-structure. Since then, the researchers have established  stage-structured population model from different angles, and the time delay considered are mostly constant forms \cite{6,7,8,9}.
%)
 
In 1992, motivated by the significant influences of the population densities on maturation length  of juvenile seals and whales found in  \cite{5},  Aiello   et al \cite{10} argued that the maturity delay of the population should be a function of the total number of populations, which  shows that due to the complexity of the ecological environment, the time lag may be adjusted continuously as the state changes, i.e., the constant time delay \cite{6,7,8,9}  can not describe the growth of the population as well as SDTD. In \cite{10},  the following SDTD single-population model is  established and analyzed:
\begin{eqnarray*}
 x_i'(t) &= &\alpha x_m(t)-\gamma x_i(t)-\alpha e^{-\gamma\tau(z(t))}x_m(t-\tau(z(t))),\\
 x_m'(t) &= &\alpha e^{-\gamma\tau(z(t))}x_m(t-\tau(z(t))) -\beta x_m^2(t).
\label{eq1}
\end{eqnarray*}
where $x_i(t)$ and $x_m(t)$ are the number of immature and mature populations at time $t$, $\alpha$ is the birth rate, $\beta$ is the internal competition coefficient of mature populations, and $\gamma$ is the mortality rate of immature populations, the SDTD $\tau(z(t))$ is taken to be an increasing differentiable bounded function of the total population $z(t) = x_i(t) + x_m(t)$. An attracting region is determined for solutions, which collapses to the unique positive equilibrium in the state-independent case.

In 2005, based on the model (\ref{eq1}),  Al-Omari et al. \cite{13} studied the following stage-dependent population model with a  SDTD:
\begin{eqnarray*}
x_i'(t) & = &R(x_m(t))-\gamma x_i(t)-R(x_m(t-\tau(z(t)))) e^{-\gamma\tau(z(t))},\\
x_m'(t) & = &R(x_m(t-\tau(z(t)))) e^{-\gamma\tau(z(t))} -\beta x_m(t).
\end{eqnarray*}
where the immature birth rate $R(y(t))$ is taken as a general function of the present mature population and the death rate for the mature one is a constant. They provided the sufficient conditions for the global stability of the extinction equilibrium  and existence of periodic solutions. Late, Magpantay et al \cite{MagpantayWu} gave an age-structured single-species population model that accounts for complex life cycles
and competition for resources limiting the transition to maturity, shown an interesting
numerical scheme and simulations to integrate the equations with significant applications. Late, by modeling the state-dependent delay as maturity period in the juvenile zooplankton population, Kloosterman et al. studied  a closed nutrient-phytoplankton-zooplankton  model
that includes size structure in the juvenile zooplankton. In 2017, Lv et al. \cite{16} studied an SDTD competitive model, where the SDTD  is taken to be an increasing differentiable bounded function of the number of its own population, and completely analyzed the global stability of the equilibria by using the comparison principle and iterative method.
%the global stability of the equilibria is completely analyzed.
For other non-predator-prey biological models with SDTD, we refer to \cite{Rezounenko2012,Rezounenko2017,Hu2014,Hu2016,LiGuo2017} and the references therein.

On the other hand,   as one of the  central goal
for ecologists, delayed predator-prey interaction has attracted many attentions. Following   pioneering works in age-structured predator-prey models by May \cite{May1975}, Hastings \cite{Hastings1983,Hastings1984}, Murdoch et al., \cite{Murdoch1987} and the stage structured predator-prey models with constant delay in \cite{6,7,8,9}, there are some recent works relate to the SDTD  predator-prey model. For example, in 2015, based on the \cite{7} constant delay model, Al-Omari \cite{14}  established and analyzed the following predator-prey model with state-dependent delay:
\begin{eqnarray*}
x'(t) & = &rx(t)(1-\frac{x(t)}{K})-ax(t)z(t)-h_1x(t),\\
y'(t) & = &bx(t)z(t)-bx\big(t-\tau(u(t))\big)z\big(t-\tau(u(t))\big)e^{-(\gamma+h_2)\tau(u(t))}-\gamma y(t)-h_2y(t),\\
z'(t) & = &bx\big(t-\tau(u(t))\big)z\big(t-\tau(u(t))\big)e^{-(\gamma+h_2)\tau(u(t))}-dz(t)-h_3z(t).
\end{eqnarray*}
where $x(t), y(t), z(t)$ represent the number of prey, juvenile predator and adult predator, respectively, $b, d, a$ are the adult predator's birth rate, mortality and capture rate, respectively, h1, h2, h3 are the prey, juvenile predator and adult predator capture rate, $\gamma$ is the juvenile predator mortality rate, $u(t)=y(t)+z(t)$, the state dependent delay $\tau(u(t))$ is a function of the total number of populations, investigated the global stability of trivial and the boundary equilibria by using Liapunov functional and LaSalle invariant principle.

In 2018, Lv et al. \cite{15} proposed and studied a predator-prey model with SDTD where the prey population is assumed to have an age structure:
\begin{equation*}
\begin{split}
x_i'(t)=&\alpha x_m(t)-\gamma x_i(t)-\alpha e^{-\gamma \tau(t)}x_m(t-\tau(t))(1-\tau'(t)),\\
x_m'(t)=&\alpha e^{-\gamma \tau(t)}x_m(t-\tau(t))-\beta x^2_m(t)-\frac{\mu x_m(t)y(t)}{x_m(t)+hy(t)},\\
y'(t)=&\frac{\mu_1 x_m(t)y(t)}{x_m(t)+hy(t)}-\delta y(t),\\
1-\tau'(t)=&\frac{k(z(t))}{k\big(z(t-\tau(t))\big)}.
\end{split}
\end{equation*}
where, $x_i(t), x_m(t)$ and $y(t)$ are the number of juvenile prey, adult feast and predator at time $t$, $\alpha$ and $\gamma$ are the birth rate and mortality of juvenile and adult bait respectively, $\beta$ is the intraspecific competition coefficient of adult prey, $\delta$ is predator mortality rate, $z(t)=x_i(t)+ x_m(t)$ is the total number of prey. For the global dynamics of the system, they discuss an attracting region which is determined by solutions, and the region collapses to the interior equilibrium in the constant delay case.

In 2018, Wang et al. \cite{11} established and analyzed a state-dependent time-delay model with the delayed time-derivative  term:
\begin{eqnarray*}
x_i'(t) & = &\alpha x_m(t)-\gamma x_i(t)-\alpha(1-\tau'(z(t))z'(t))e^{-\gamma\tau(z(t))}x_m(t-\tau(z(t))), \\
x_m'(t) & = &\alpha(1-\tau'(z(t))z'(t))e^{-\gamma\tau(z(t))}x_m(t-\tau(z(t))) -\beta x_m^2(t).
\end{eqnarray*}
This model is clearly different from the previous models in the sense that it includes the correction term $1-\tau'(z(t))z'(t)$ in the maturity rate. Permanence of the system was analyzed, and explicit bounds for the eventual behaviors of the immature and mature populations are established in  \cite{11}, while the global stability was not discussed.

In this paper, following the research method of Wang et al. \cite{11}, based on the age structure model, we consider a SDTD predator-prey model where the predators was divided into immature and mature, mature delay is a monotonically increasing continuous differentiable bounded function that depends on the number of adult predators. In other words, if the number of adult predators is large, the maturity period is longer, and the population size is reduced as slow growth of adult predators.   Our aim is to conduct a qualitative analysis of the model to study the persistence, extinction and stability of the population.

Our paper is organized as follows. In Sect. 2, we establish a state-dependent time-delay predator-prey model with the delay time-derivative  term. In Sect. 3, we discuss the positivity and boundedness of solutions. In Sect. 4, we give the necessary and sufficient conditions for the permanence and necessity of the population. In Sect. 5, we discuss the linearized stability of all equilibria of the model. In Sect. 6, we prove the global attractiveness of positive equilibrium of the model. The summary and discussion are presented in Sect. 7.

\section{Model Derivation}
\label{sec:1}
Motivated by the ideas in \cite{7,9,MagpantayWu,11}, we consider the growth of predator through immature and mature two stages. In order to distinguish immature individuals, $y_j(t)$, from mature ones, $y(t)$, we introduce a threshold age $\tau(y(t))$, which is the maturation time for an immature individual that matures at time $t$ depending on the number of mature predator, $y(t)$. Let $\rho(t,a)$ be the density of predator of age $a$ at time $t$. Then the number of immature, $y_j(t)$, and  mature $y(t)$ is given by
$$y_j(t)=\int_{0}^{\tau(y(t))}\rho(t,a)da ~~~~\mbox{and} ~~~~y(t)=\int_{\tau(y(t))}^{\infty}\rho(t,a)da, $$ respectively.
The dynamics of predator with age structure can be represented (see \cite{12,18}) by the following partial differential equations:
\begin{equation}\label{M1}
\begin{split}
\frac{\partial \rho(t,a)}{\partial t}+\frac{\partial \rho(t,a)}{\partial a} &=-d_j\rho(t,a),~~~if~a\leq \tau(y(t)), \\
\frac{\partial \rho(t,a)}{\partial t}+\frac{\partial \rho(t,a)}{\partial a} &=-d\rho(t,a),~~~if~a>\tau(y(t)).
\end{split}
\end{equation}
where each individual from $y_j(t)$ dies at a constant rate $d_j$ and that from $y(t)$ at a constant rate $d$.

Taking the derivatives of $y_j(t)$ and $y(t)$, respectively, and combining with (1), we get
\begin{eqnarray*}
\frac{dy_j(t)}{dt} &= &\rho(t,0)-d_jy_j(t)-[1-\tau'(y(t))y'(t)]\rho(t,\tau(y(t))),\\
\frac{dy(t)}{dt} &= &[1-\tau'(y(t))y'(t)]\rho(t,\tau(y(t)))-\rho(t,\infty)-dy(t).
\end{eqnarray*}

It is necessary to note that a prime refers to differentiation with respect to $y$, and a dot indicates differentiation with respect to time $t$, namely, $ d\tau(y(t))/dt=\tau'(y)\cdot y'(t).$

Taking $\rho(t,\infty)$ as zero since no one can live forever.
We assume that the birth rate of predator is $n$ and its functional response function is $f(x(t),y(t))$, where $x(t)$ is the total number of prey, so the term $\rho(t,0)=nf(x(t),y(t))y(t)$. Therefore, for $t\geq {\tau_M}=max\{\tau(y(t))\} $       we obtain
$$\rho(t,\tau(y(t)))=\rho(t-\tau(y(t)),0)e^{-d_j\tau(y)}=nf(x(t-\tau(y),y(t-\tau(y)))y(t-\tau(y))e^{-d_j\tau(y)}.$$
Further we assume that the growth of prey obays logistic growth, then
we have the following  predator-prey model with a state-dependent maturation delay for the predator:
\begin{equation}\label{M2}
\begin{split}
x'(t)=&rx(t)(1-\frac{x(t)}{K})-f(x(t),y(t))y(t),\\
y_j'(t)=&-n(1-\tau'(y)y'(t)) e^{-d_j\tau(y)}f(x(t-\tau(y),y(t-\tau(y)))y(t-\tau(y))\\
&+nf(x(t),y(t))y(t)-d_jy_j(t),\\
y'(t)=&n(1-\tau'(y)y'(t)) e^{-d_j\tau(y)}f(x(t-\tau(y),y(t-\tau(y)))y(t-\tau(y))-dy(t),
\end{split}
\end{equation}
where $r$ and $K$ represent the specific growth rate of the prey and environmental carrying capacity, respectively.
%Model (\ref{M2}) is different from the previous SDTD equations in the sense that it includes the delay time-derivative  term $1-\tau'(y)y'(t)$ in the maturity rate.

From a biological point of view, it is natural that $t-\tau(y(t))$ is a strictly increasing function of $t$, i.e., the possibility of mature individuals becoming immature only by birth. Assume that $q(s)$ is the developmental proportion at time $s$, when an immature individual moves to the mature state from $t-\tau(y(t))$ to $t$, the cumulative rate of development $q$ should be equal to one, namely,
$$\int_{t-\tau(y(t))}^{t}q(s)ds=1.$$
Taking the derivative with respect to $t$, we have
$$1-\tau'(y(t))y'(t)=\frac{q(t)}{q\big(t-\tau(y(t))\big)}>0,$$
implying that $t-\tau(y(t))$ is a strictly increasing function of $t$ and the variation of maturity time delay $\tau(y(t))$ is bounded by one.

Furthermore, from the biological point of view, we give the following hypotheses for the model (\ref{M2}):

  \begin{itemize}
\item [$(A_1)$] Parameters $r, K, n, d_j, d$  are all positive constants;

\item [$(A_2)$] The state-dependent maturity time delay $\tau (y)$ is an increasing continuously  differentiable bounded function of the mature predator population $y(t)$, i.e.,  $\tau'(y)\geq 0$, and
$0\leq\tau_m\leq\tau(y)\leq\tau_M<+\infty$ with $\tau(+\infty)=\tau_M, \tau(0)=\tau_m.$
\item [$(A_3)$] The functional response function $f(x,y)$  satisfies the following conditions:
\end{itemize}
\begin{itemize}
\item [ (i)]$f(x,y)$ is continuously differentiable; $f(x,y)>0$ for $\forall \ x,y\in(0, \infty)$,  and $f(x,y)=0$ if and only if $x = 0$; $|f_x'(0,0)|<+\infty$, where $f_x'(0,0)$ denotes $\frac{\partial f}{\partial x}(x,y)\big|_{x=0,y=0}$.

\item[(ii)] $f(x,y)$ is increasing of $x$ and decreasing of $y$.

\end{itemize}

In fact, in the literature, most of the commonly used functional response functions $f(x,y)$ satisfies the condition $(A_3)$. We can summarize the forms in the following:
\begin{itemize}
\item [(1)] Linear type \cite{11}: $f(x,y)=bx$;
\item [(2)] Nonlinear type \cite{29}: $f(x,y)=bx^k,~~k>1$; 
\item [(3)] Holling type I (Blackman type): \cite{27}
$f(x,y)=\left\{
\begin{array}{rcl}
ax      &      & {x     \leq    b}\\
ab     &      & {x     >b.     }\\
\end{array}
\right. $
\item [(4)] Holling type II (Michaelis-Menten type) \cite{27}: $f(x,y)=\frac{bx}{1+bhx}$;
\item [(5)] Holling type III \cite{27}: $f(x,y)=\frac{bx^2}{1+bhx^2}$;
\item [(6)] Saturation type \cite{28} $f(x,y)=\frac{bx^k}{1+bhx^k},\ \ k>1$;
\item [(7)] Ivlev type \cite{24}: $f(x,y)=b(1-e^{-cx})$;
\item [(8)] Beddington-DeAngelis type \cite{26,DA}: $f(x,y)=\frac{bx}{1+k_1x+k_2y}$;
\item [(9)] Crowley-Martin type \cite{25}: $f(x,y)=\frac{bx}{1+k_1x+k_2y+k_1k_2xy}$.
\end{itemize}

\section{Positivity and Boundedness}
In this section, we shall address the positivity and boundedness of the solution of system (\ref{M2}). From the standpoint of biology, positivity means that the species persists, i.e., the populations can not be extincted. Boundedness may be viewed as a natural restriction to growth as a result of limited resources in an closed environment.

Denote $C:= C([-\tau_M, 0), R^3)$. For $\phi=(\phi_1,\phi_2,\phi_3)\in C$, define $\Arrowvert\phi\Arrowvert=\sum_{i=1}^3\Arrowvert\phi_i\Arrowvert_\infty$, where
$$\Arrowvert\phi_i\Arrowvert_\infty=max_{\theta\in[-\tau_M,0]}|\phi_i(\theta)|.$$
Then $C$ is a Banach space and $C_+=\{\phi\in C:\phi_i(\theta)\geq0, i\in\{1,2,3\}, \theta\in[-\tau_M,0]\}$ is a normal cone of $C$ with nonempty interior in $C$. The initial conditions for the system (\ref{M2}) are
\begin{equation}\label{M3}
%\begin{split}
x(t) =\phi_1(t)\geq0,\ \  y_j(t) =\phi_2(t)\geq0,\  \
y(t) =\phi_3(t)\geq0 
 ~~~\mbox{for~all}    -\tau_M\leq t\leq 0,
% \end{split}
\end{equation}
with positive $\phi_1(0), \phi_2(0), \phi_3(0)$ and
$$\phi_2(0)=\int_{-\tau(y(0))}^{0}f\big(x(s),y(s)\big)y(s)e^{d_js}ds$$
presenting the size of the immature population surviving to time $t = 0$, where $\tau(y(0))$ is the maturation time at $t = 0$.

The following theorem demonstrates that the solutions of (\ref{M2}) are positive and bounded.

\begin{theorem}\label{positivenewss}
 Assume that the initial condition (\ref{M3}) holds,then every solution\\
$(x(t),y(t),y_j(t))$ of system (\ref{M2}) is positive for all $t>0$.
\end{theorem}

\begin{proof}
Clearly, $x(t)>0$ for all $t>0$ ({\bf why? seems it's not obvious, better to give some reason}).

Now to prove the positivity of $y(t)$ as a solution of (\ref{M2}). Assume that,there exists $t^*>0$, such that $y(t^*)=0$ and $y(t)>0$ for every $t\in(0,t^*)$. Then by the initial conditions  (\ref{M3}) we have
$
y'(t)\geq-dy~~~ \forall t\in[-\tau_M,t^*].
$
Thus we get
$
y(t)\geq y_0e^{-dt}>0$ for $ \forall t\in[-\tau_M,t^*]
$.
Let $t=t^*$, then we have
$
0=y(t^*)\geq y_0e^{-dt^*}>0
$, 
which is a contradiction. So no such $t^*$ exists, and we obtain the results.

Next we  prove the positivity of $y_j(t)$ as a solution of (\ref{M2}).

Integrating the third equation of the system (\ref{M2}), we have
\begin{equation}\label{M4}
\begin{split}
y_j(t)=&e^{-d_j t}\bigg[\phi_3(0)+\int_{0}^{t} e^{d_j s}nf(x(s),y(s))y(s)ds-\int_{-\tau(y(0))}^{t-\tau(y(t))} nf(x(s),y(s))y(s)e^{d_j s}ds\bigg]\\
=&\int_{t-\tau(y)}^{t}nf(x(s),y(s))y(s)e^{d_j(t-s)}ds.
\end{split}
\end{equation}

By the positivity of $x(t),y(t),\tau(t)$, we have  $y_j(t)>0,\  \forall t>0$. The proof is complete.
\end{proof}

In proving our main results, a comparison principle will be used. As we know, the comparison principles do not always hold for SDTD equations, which depends very much on how the delay term appears in the equations. The comparison principle of SDTD equations without delay time-derivative  term is discussed in \cite{16}. Now we extend the result to the system with delay time-derivative  term.
\begin{lemma}
Let $y_1(t)$ be the solution of
$$y_1'(t)=\big[1-\tau'(y_1)y_1'(t)\big]e^{-d_j\tau(y_1)}f(t-\tau(y_1))-dy_1(t),~~t>0$$
and $y_2(t)$ be a function satisfying
\begin{equation}\label{B1}
y_2'(t)\leq\big[1-\tau'(y_2)y_2'(t)\big]e^{-d_j\tau(y_2)}f(t-\tau(y_2))-dy_2(t),~~t>0
\end{equation}
and  $y_2(s)\leq y_1(s)$ for all $s\in[-\tau_M,0]$, where $f(\cdot)$ is a continuous function. Then   $y_2(t)\leq y_1(t)$ holds true for all $t>0$.
\end{lemma}
\begin{proof}
{\bf I think the lemma has problem, and the proof}

With the assumption $y_2(s)<y_1(s)$ for all $s\in[-\tau_M,0]$,
first we claim that $y_2(t)<y_1(t)$ for all $t>0$. If this  is false, there must exist some $t_0> 0$ such that $y_2(t)<y_1(t)$, $t\in[-\tau_M,t_0)$ and $y_2(t_0)=y_1(t_0)$. It follows that $y_2'(t_0)\geq y_1'(t_0)$. But
\begin{equation*}
\begin{split}
y_1'(t_0)&=\big[1-\tau'(y_1(t_0))y_1'(t_0)\big]e^{-d_j\tau(y_1(t_0))}f(t-\tau(y_1(t_0)))-dy_1(t_0)\\
&\geq \big[1-\tau'(y_2(t_0))y_2'(t_0)\big]e^{-d_j\tau(y_2(t_0))}f(t-\tau(y_2(t_0)))-dy_2(t_0)\\
&\geq y_2'(t_0)
\end{split}
\end{equation*}
%because $y_2(t_0-\tau(y_2(t_0)))<y_1(t_0-\tau(y_1(t_0)))$ and $y_2(t_0)=y_1(t_0)$. This contradiction proves the result in this case.

\vspace{0.2in}

For the general case, let $\varepsilon>0$ and $y_\varepsilon(t)$ be the solution of
$$y_1'(t)=\big[1-\tau'(y_1)y_1'(t)\big]e^{-d_j\tau(y_1)}f(t-\tau(y_1))-dy_1(t)+\varepsilon$$
corresponding to initial data $y_\varepsilon(t)=y_1(t)+\varepsilon$, $t\in[-\tau_M,0)$. Following the previous result, we can conclude that $y_\varepsilon(t)>y_2(t)$ for all $t>0$ for which $y_\varepsilon(t)$ is defined. It can be shown that for sufficiently small $\varepsilon>0$, the solution $y_\varepsilon(t)\to y_1(t)$ as $\varepsilon\to0$ for all $t\geq-\tau_M$. Consequently, $y_1(t)=\lim_{\varepsilon\to0}y_\varepsilon(t)\geq y_2(t).$
\end{proof}

For the comparison principles, the other reversed inequality follows analogously, and will be used later. {\bf Not sure what you want to claim for this part: Furthermore, an differential inequality of the form (\ref{B1}), which holds only for t above some value, say $t_1$, and not for all $t>0$, will be often used in applications of these comparison results. That is the initial time is simply thought of as $t_1$ rather than $0$, and $y_2(t)\leq y_1(t)$ is arranged to hold for $t\leq t_1$ by appropriate definition of $y_1(t)$ for values of $t\leq t_1$. In the interests of clarity, this latter case in detail will not be always explained.
}

In the next theorem we give boundedness results.

\begin{theorem}
 Assume that the initial condition (\ref{M3}) holds, then every solution $(x(t),y(t),y_j(t))$ of system $(2)$ is eventually uniformly bounded.
 \end{theorem}
\begin{proof}Define a function as follows:
$$V(t)=nx(t)+y(t)+y_j(t).$$
Calculating the time derivative of $V(t)$ along the solutions of system (\ref{M2}), we obtain
\begin{equation*}
\begin{split}
V'(t) & =nrx(t)(1-\frac{x(t)} {K})-dy-d_j y_j\\
& =-\min\{d_j, d\} V+n(\min\{d_j, d\}+r)x-\frac {nr}{K} x^2+(\min\{d_j, d\}-\max\{d_j, d\})y\\
& \leq -\min\{d_j, d\}V+M.
\end{split}
\end{equation*}
where $M>0$ is the maximum of the quadratic function $n(\min\{d_j, d\}+r)x-\frac {nr}{K} x^2$. Therefore, $\limsup_{t \rightarrow \infty}V(t)\leq M/min\{d_j, d\}$ and the solution of system (\ref{M2}) is eventually uniformly bounded.

In particular, according to the first equation of system (\ref{M2}), we have
$$x'(t)\leq rx(t)\big(1-\frac{x(t)}{K}\big)$$
By comparison principle\cite{CP}, we can get
$$\limsup_{t\to\infty} x(t)\leq K$$
The proof is complete.
\end{proof}

\begin{remark}
Due to the existence of the delay time-derivative  term $1-\tau'(y(t))y'(t)$, the method of proving boundedness in \cite{6} is not applicable. Here, the method of constructing the $V$ function is used to subtly prove the eventually boundedness of $x(t),y(t),y_j(t)$.
\end{remark}

\section{Permanence and Extinction}
It is clear that system (\ref{M2}) has a trivial equilibrium $E_0(0,0,0)$ and a predator extinction equilibrium $E_1(K,0,0)$. Mathematically, permanence is equivalent to the existence of positive equilibrium$E^*(x^*,y^*,y_j^*)$. The following results give the necessary and sufficient conditions for the permanence/extinction of the system (\ref{M2}).

\begin{theorem} \label{TH1} 
System (\ref{M2}) is permanent if and only the $ne^{-d_j\tau(0)}f(K,0)> d$ holds.
\end{theorem}
To prove this theorem, we engage the persistence theory by Hale and Waltmann \cite{19} for infinite dimensional systems (see \cite{20} as well), in the following:

Consider a metric space X with metric d. T is a continuous semiflow on X,i.e., a continuous mapping $T: [0,\infty)\times X\to X$ with the following properties:
$$T_t\circ T_s=T_{t+s},~~~~t,s\geq0;~~~~~T_0(x)=x,~~x\in X.$$
Here $T_t$ denotes the mapping from X to X given by $T_t(x)=T(t,X)$. The distance $d(x,Y)$ of a point $x\in X$ from a subset Y of X is defined by
$$d(x,Y)=\inf_{y\in Y} d(x,y).$$
Recall that the positive orbit $\gamma^+(x)$ through $x$ is defined as $\gamma^+(x)=\cup_{t\geq0} \{ T(t)x\}$, and its $\omega$-limit set $\omega (x)= \cap_{\tau \geq 0} CL \cup_{t\geq \tau } \{ T(t)x\}$, where CL means closure.  Define $W^s(A)$ the stable set of an compact invariant set A as
$$W^s(A)=\{ x: x\in X,\omega (x)\neq \phi, \omega(x)\subset A\};$$
 define $\widetilde{A_\partial}$ the particular invariant set of interest as
 $$\widetilde{A_\partial} = \bigcup_{x\in A} \omega (x).$$

\begin{lemma}(\cite{19})
Suppose T(t) satisfies $(H_1)$:

$(H_1)$. Assume X is the closure of open set $X^0$; $\partial X^0$ is nonempty and is the boundary of $X^0$; and the $C^0$-semigroup $T(t)$ on X satisfies
$$T(t): X^0\times X^0,~~~~~T(t): \partial X^0\to \partial X^0.$$

and
\begin{itemize}
\item [(i)] there is a $t_0\geq0$ such that $T(t)$ is compact for $t>t_0$;
\item [(ii)] T(t) is point dissipative in X;
\item [(iii)] $\widetilde{A_\partial}$ is isolated and has an acyclic covering M.
\end{itemize}
Then $T(t)$ is uniformly persistent iff for each $M_i\in M$, $W^s(M_i)\cap X^0=\emptyset$.
\end{lemma}
To prove the theorem (\ref{TH1}), first we have the following claims.
\begin{claim}
{\bf A}: If $ne^{-d_j\tau(0)}f(K,0)> d$, then the system (\ref{M2}) is permanent.
\end{claim}
\begin{proof}
We begin by showing the claim holds true for the
system (\ref{M5}),
\begin{equation}\label{M5}
\begin{split}
x'(t) &=rx(t)(1-\frac{x(t)}{K})-f(x(t),y(t))y(t),\\
y'(t) &=n(1-\tau'(y)y'(t)) e^{-d_j\tau(y)}f[x(t-\tau(y),y(t-\tau(y))]y(t-\tau(y))-dy(t),
\end{split}
\end{equation}
which is a subsystem of system (\ref{M2}).
As the first step, we verify that the boundary plan $R_+^2=\{(x,y): x\geq0,y\geq0\}$ repel the positive solutions to systems (\ref{M5}) uniformly.

Let $C^{+}([-\tau_M,0],R_+^2)$ denote the space of continuous functions mapping  $[-\tau_M,0]$ into $R_+^2$.  We choose
\begin{equation*}
\begin{split}
C_1=\{(\phi_1,\phi_2)\in C^+( [-\tau_M,0], R_+^2): \phi_1(\theta) \equiv 0,\phi_2(\theta)\geq 0, \theta\in[-\tau_M,0]\},\\
C_2=\{(\phi_1,\phi_2)\in C^+( [-\tau_M,0], R_+^2): \phi_1(\theta) >0,\phi_2(\theta) \equiv 0, \theta\in[-\tau_M,0]\}.
\end{split}
\end{equation*}
Denote  $C=C_1  \cup C_2$, $X=C^+( [-\tau_M,0], R_+^2)$, and $X_0=Int C^+( [-\tau_M,0], R_+^2)$; then $C=\partial X^0$. It is easy to see that the system (\ref{M2}) possesses two constant solutions in $C=\partial X^0$: $\widetilde{E_0}\in C_1, \widetilde{E_1}\in C_2$ with
\begin{equation*}
\begin{split}
\widetilde{E_0}=&\{(\phi_1,\phi_2)\in C^+( [-\tau_M,0], R_+^2): \phi_1(\theta)=\phi_2(\theta)\equiv 0,\theta\in[-\tau_M,0]\},\\
\widetilde{E_1}=&\{(\phi_1,\phi_2)\in C^+( [-\tau_M,0], R_+^2): \phi_1(\theta)\equiv K,\phi_2(\theta)\equiv 0, \theta\in[-\tau_M,0]\}.
\end{split}
\end{equation*}
We verify below that the conditions of Lemma $2$ are satisfied. By the definition of $X^0$ and $\partial X^0$, it is easy to see that condition (i) and (ii) of Lemma $2$ are satisfied for the system (\ref{M5})and that $X^0$ and $\partial X^0$ are invariant. Hence $(H_1)$ is satisfied.

Consider condition (iii) of Lemma $2$.  We have
$$x'(t)|_{(\phi_1,\phi_2)\in C_1}\equiv0,$$
thus  $x(t)|_{(\phi_1,\phi_2)\in C_1}\equiv0$ for all $t\geq 0$. Consequently
we have
$y'(t)|_{(\phi_1,\phi_2)\in C_1}=-dy(t)\leq0,$
implying all the points in $C_1$ approach $\widetilde{E_0},$ i.e., $C_1=W^s(\widetilde{E_0}).$ Similarly we have
$C_2=W^s(\widetilde{E_1}).$ Hence $\widetilde{A_\partial}=\widetilde{E_0}\cup \widetilde{E_1}$ and clearly it is isolated. Noting that $C_1\cap C_2=\emptyset$, it follows from these structural features that the flow in $\widetilde{A_\partial}$ is acyclic, satisfying condition (iii) in Lemma $2$.

Now we show that $W^s(\widetilde{E_i})\cap X^0=\emptyset, (i=0,1).$ By Theorem (\ref{positivenewss}), we have $x(t),y(t)>0$ for all $t>0$. Assume $W^s(\widetilde{E_0})\cap X^0\neq \emptyset$, i.e., there exists a positive solution $(x(t),y(t))$ with $\lim_{t\rightarrow \infty}(x(t),y(t))=(0,0)$, by $(A3)$-(i), we have $\lim_{x\to0,y\to0}\frac{f(x(t),y(t))}{x(t)}=f'_x(0,0)<+\infty$. Thus $\lim_{x\to0,y\to0}\frac{f(x(t),y(t))y(t)}{x(t)}=0$. Then from the first equation of (6), we have
$$\frac{d(ln x(t))}{dt}=r(1-\frac{x(t)}{K})-\frac{f(x(t),y(t))y(t)}{x(t)}>\frac{r}{2}$$
for all sufficiently large $t$. Hence we have $\lim_{t \rightarrow \infty}x(t)=+\infty,$ contradicting $\lim_{t \rightarrow \infty}x(t)=0$; this proves $W^s(\widetilde{E_0})\cap X^0=\emptyset$.

Now we verify $W^s(\widetilde{E_1})\cap X^0=\emptyset$. If it is not the case,
i.e., $W^s(\widetilde{E_1})\cap X^0\neq \emptyset$. Then there exists a positive solution $(x(t),y(t))$ to system (6) with $\lim_{t\rightarrow \infty}(x(t),y(t))=(K,0)$, and for sufficiently small positive constant $\varepsilon$ with $ne^{-d_j\tau(\varepsilon)}f(K-\varepsilon,\varepsilon)>d$, there exists a positive constant $T=T(\varepsilon)$ such that
$$x(t)>K-\varepsilon>0, y(t)<\varepsilon~~~\mbox{for~~all}~~ t\geq T.$$
Consider the function
$$V=y(t)+d\int_{t-\tau(y)}^{t} y(s)ds.$$
Based on hypotheses $A_3$, we have
\begin{equation*}
\begin{split}
\frac{dV}{dt}|_{(2)}=&ne^{-d_j\tau(y)}(1-\tau'(y)y'(t))f[x(t-\tau(y)),y(t-\tau(y))]y(t-\tau(y))-dy(t)\\
&+d\big[y(t)-y(t-\tau(y))(1-\tau'(y)y'(t))\big]\\
=&y(t-\tau(y))(1-\tau'(y)y'(t))\big[ne^{-d_j\tau(y)}f[x(t-\tau(y)),y(t-\tau(y))]-d\big]\\
>&y(t-\tau(y))(1-\tau'(y)y'(t))\big[ne^{-d_j\tau(\varepsilon)}f(K-\varepsilon,\varepsilon)-d\big]\\
>&0.
\end{split}
\end{equation*}
Hence we have $\lim_{t \rightarrow \infty}y(t)=+\infty,$ contradicting with $\lim_{t \rightarrow \infty}y(t)=0$, this proves ${W^s(\widetilde{E_1}}\cap X^0)=\emptyset$.

Therefore we know that the system (\ref{M5}) satisfies all the conditions in Lemma 1, thus $(x(t),y(t))$ is uniformly persistent, i.e., there exists positive constants $\varepsilon$ and $T=T(\varepsilon)$ such that $x(t),y(t)\geq\varepsilon$ for all $t\geq T$; In addition, from Theorem $2$,we have that $(x(t),y(t))$ is eventually bounded, implying the permanence of the system (\ref{M5}). Obviously $y_j(t)$ is permanent from (\ref{M4}), so the permanence of system (\ref{M2}) is straightforward.
\end{proof}

The following claim is need for the proof of necessity.

\begin{claim}
{\bf B}: $\lim_{t \rightarrow \infty}(x(t),y(t),y_j(t))=(K,0,0)$ holds true iff $ne^{-d_j\tau(0)}f(K,0)\leq d$.
 \end{claim}
\begin{proof}
By the first equation of system (\ref{M2}), $x(t)$ is always decreasing when $x(t)>K$. We can show that if there exists some $t_0>0$ such that $x(t_0)<K$, then $x(t)<K$ for all $t>t_0$. Otherwise there must exist some $t_1>t_0$ such that $x(t_1)=K$ and $x'(t_1)\geq0$. This is impossible. Hence, there are two possible cases, either

(1) $x(t)>K$ and $x(t){\rightarrow K}$ as $t \rightarrow \infty$, or

(2) there exists some $t_0>0$ such that $x(t_0)<K$.

For the first of these cases, we only need to prove that $\lim_{t \rightarrow \infty} y(t)=0$, since this implies   $\lim_{t \rightarrow \infty} y_j(t)=0$. Integrating the equation for $x(t)$ in (\ref{M2}), we have
\begin{equation*}
\begin{split}
x(t)-x(0) &=\int_{0}^{t}rx(s)(1-\frac{x(s)} {K})ds-\int_{0}^{t}\ f(x(s),y(s))y(s)ds\\
&<\int_{0}^{t} \underbrace {rx(s)(1-\frac{x(s)} {K})}_{x(s)\geq K}ds-\int_{0}^{t}\ f(K,y(s))y(s)ds
\end{split}
\end{equation*}
for all $t\geq 0$, and then
$$\int_{0}^{t}\ f(K,y(s))y(s)ds<x(0)-x(t)+\int_{0}^{t} \underbrace {rx(s)(1-\frac{x(s)} {K})}_{\leq 0}ds<x(0).$$
By the boundedness of $y(t)$,then $\int_{0}^{t}y(s)ds$ is bounded for all $t\geq t_0$,and this implies $\lim_{t \rightarrow \infty} y(t)=0$.

For the second of these cases, consider the function
$$V=y(t)+d\int_{t-\tau(y)}^{t} y(s)ds$$
Then based on hypotheses $A3-ii$, for all $t\geq t_0+\tau_M$, we have
\begin{equation*}
\begin{split}
\frac{dV}{dt}|_{(2)}=&ne^{-d_j\tau(y)}(1-\tau'(y)y'(t))f(x(t-\tau(y)),y(t-\tau(y)))y(t-\tau(y))-dy(t)\\
&+d\big[y(t)-y(t-\tau(y))(1-\tau'(y)y'(t))\big]\\
=&y(t-\tau(y))(1-\tau'(y)y'(t))\big[ne^{-d_j\tau(y)}f(x(t-\tau(y)),y(t-\tau(y)))-d\big]\\
<&y(t-\tau(y))(1-\tau'(y)y'(t))\big[ne^{-d_j\tau(0)}f(K,0)-d\big].
\end{split}
\end{equation*}
When $ne^{-d_j\tau(0)}f(K,0)\leq d$, $\frac{dV}{dt}|_{(2)}<0$ yields
$\lim_{t \rightarrow \infty}(x(t),y(t),y_j(t))=(K,0,0)$i from the positivity of $y(t)$.

The prove the necessary condition for $\lim_{t \rightarrow \infty}(x(t),y(t),y_j(t))=(K,0,0)$, assume in contrast, i.e., $ne^{-d_j\tau(0)}f(K,0)>d$,
there exists a positive equilibrium $(x^*,y^*,y_j^*)$ in the system (\ref{M2}),
contradicting with $\lim_{t \rightarrow \infty}(x(t),y(t),y_j(t))=(K,0,0)$ for all solution $(x(t),y(t),y_j(t))$. Hence there must be $ne^{-d_j\tau(0)}f(K,0)\leq d$ which is the sufficient condition in Theorem $3$.

To show the necessity of Theorem $3$, we assume, by contrast, i.e., \\
$ne^{-d_j\tau(0)}f(K,0)\leq d$; then by Claim {\bf B}, $x(t)\to K, y(t)\to0$ as $t\to\infty$, which contradict the permanence of (\ref{M2}). This ends of the prove in Theorem $3$. \end{proof}

\begin{remark}
Due to the existence of delay time-derivative  term $1-\tau'(y)y'(t)$ in our model, the comparison principle cannot be used directly. We overcome the difficulty by constructing some Lyapunov-like function. Here
Theorems 3 is an extension of \cite[Theorem 3.2]{6} and The claim {\bf B} extends \cite[Theorem 3.1]{6}.
\end{remark}

\section{Linearized Stability}
In this section,we study the linearized stability of the equilibria $E_0$, $E_1$ and $E^*$ in the system (\ref{M5}).
Different from the linearization for constant delay system, here
the delay is a function depending on the state variable $y$, thus linearizing the system with SDTD is not completely straightforward.
Adopting the ``freezing the delay '' idea introduced in \cite{21},
we linearize the system (\ref{M5}) first.

By equation (\ref{M4}), we can get the same conclusions for system (\ref{M2}) and System (\ref{M5}).
Let $E^*(x^*,y^*)$ be an arbitrary equilibrium. we can get the linearized system of (\ref{M5}) as :
\begin{equation}\label{M6}
\begin{split}
x'(t) &=Ax(t)-By(t),\\
y'(t) &=Cx(t-\tau(y^*))+(\eta-d)y(t)+Dy(t-\tau(y^*)).
\end{split}
\end{equation}
with
\begin{align*}
A&=r-\frac{2r}{K}x^*-f^*_xy^*, &           B&=f^*+f^*_yy^*,  \\
C&=ne^{-d_j\tau(y^*)}f^*_xy^*, &     D&=ne^{-d_j\tau(y^*)}(f^*+f^*_yy^*),\\
\eta&=ne^{-d_j\tau(y^*)}\tau'(y^*)f^*y^*(d-d_j).
\end{align*}
where $f^*=f(x,y)\Big|_{x=x^*,y=y^*}$, $f_x^*=\frac{\partial f}{\partial x}(x,y)\Big|_{x=x^*,y=y^*}$, $f_y^*=\frac{\partial f}{\partial y}(x,y)\Big|_{x=x^*,y=y^*}$.\\
This leads to the following characteristic equation:
\begin{equation*}
(\lambda-A)\big(\lambda+d-\eta-De^{-d_j\tau(y^*)}\big)+BCe^{-\lambda\tau(y^*)}=0.
\end{equation*}

\begin{theorem}
The trivial equilibrium $E_0 = (0, 0)$  is unstable.
\end{theorem}
\begin{proof}
For the trivial equilibrium $E_0 = (0, 0)$, we have

$$A=r, \ B=C=D=\eta=0.$$

The characteristic equation
\begin{equation}\label{M7}
(\lambda-r)(\lambda+d)=0
\end{equation}
Obviously, there is a positive real root $\lambda=r$ in (\ref{M7}). Hence $E_0 = (0, 0)$  is unstable.
\end{proof}
\begin{theorem}
The predator extinction equilibrium $E_1(K,0)$ is
\begin{itemize}
\item[(i)]  unstable if $ne^{-d_j\tau(0)}f(K,0)>d$;
\item[(ii)]  linearly neutrally stable if $ne^{-d_j\tau(0)}f(K,0)=d$;
\item[(iii)] locally asymptotically stable if $ne^{-d_j\tau(0)}f(K,0)<d$.
\end{itemize}
\end{theorem}

\begin{proof}
 For the extinction equilibrium $E_1 = (K, 0)$, we have

$$A=-r, \ B=f(K,0),\ D=ne^{-d_j\tau(0)}f(K,0),\ C=\eta=0.$$

The characteristic equation
\begin{equation}\label{M8}
(\lambda+r)\big[\lambda+d-ne^{-d_j\tau(0)}f(K,0)e^{-\lambda\tau(0)}\big]=0
\end{equation}\label{M9}
Obviously, the stability of $E_1$ is determined by the roots in
\begin{equation}\label{M9}
G(\lambda)=\lambda+d-ne^{-d_j\tau(0)}f(K,0)e^{-\lambda\tau(0)}=0
\end{equation}
\begin{itemize}
\item[(i)] When $ne^{-d_j\tau(0)}f(K,0)>d$,  we have $G(0)=d-ne^{-d_j\tau(0)}f(K,0)<0$, and $G(+\infty)=+\infty$. Hence $G(\lambda)$ has at least one positive root and $E_1$ is unstable.
\item[(ii)] As $ne^{-d_j\tau(0)}f(K,0)=d$,  $G(\lambda)=\lambda+d-de^{-\lambda\tau(0)}$, so $\lambda=0$ is a root of $G(\lambda)=0$. Furthermore, since $G'(\lambda)=1+\tau(0)de^{-\lambda\tau(0)}$,  we have $G'(0)>0$. Then, the root $\lambda=0$ is simple.

Denote all the other roots as $\lambda=u+iv$, then $u,v$ must satisfy
$$(u+d)^2+v^2=d^2e^{-2u\tau(0)}.$$
Thus $u\leq0$, i.e., all the other roots have real nonpositive parts. Therefore $E_1$ is linearly neutrally stable.
\item[(iii)]  If $ne^{-d_j\tau(0)}f(K,0)<d$,  i.e.
$$d-nf(K,0)e^{-d_j\tau(0)}e^{-\lambda\tau(0)}>0$$
Then $G(\lambda)=0$ implies that
$$\lambda+d=nf(K,0)e^{-d_j\tau(0)}e^{-\lambda\tau(0)}.$$
Assume $Re(\lambda)\geq0$, then
$$|\lambda+d|>d>nf(K,0)e^{-d_j\tau(0)}e^{-\lambda\tau(0)}$$
yields contradiction.
This shows that all roots of $G(\lambda)=0$ must have negative real parts, and therefore $E_1$ is locally asymptotically stable.
\end{itemize}
\end{proof}

  Combining the result in Theorem 4, we have:
\begin{corollary}
The equilibrium $E_1(K,0)$ of system (\ref{M5}) is globally asymptotically stable iff $ne^{-d_j\tau(0)}f(K,0)<d$  holds true.
\end{corollary}

{\bf To discuss the stability of the positive equilibria $E_2$,
For simplicity, in the rest of the manuscript,
we chose}
$$f(x,y)=\frac{bx}{1+k_1x+k_2y},$$ for simplicity, in the rest of the manuscript.
Then the system (\ref{M5}) becomes
\begin{equation}\label{N1}
\begin{split}
x'(t) &=rx(t)(1-\frac{x(t)}{K})-\frac{bx(t)y(t)}{1+k_1x(t)+k_2y(t)},\\
y'(t) &=n(1-\tau'(y)y'(t)) e^{-d_j\tau(y)}\frac{bx(t-\tau(y))y(t-\tau(y))}{1+k_1x(t-\tau(y))+k_2y(t-\tau(y))}-dy(t).
\end{split}
\end{equation}
and
\begin{align*}
A&=r-\frac{2r}{K}x^*-\frac{(1+k_2y^*)by^*}{(1+k_1x^*+k_2y^*)^2}, &           B&=\frac{(1+k_1x^*)bx^*}{(1+k_1x^*+k_2y^*)^2},\\
C&=\frac{nby^*e^{-d_j\tau(y^*)}(1+k_2y^*)}{(1+k_1x^*+k_2y^*)^2}, &     D&=\frac{nbx^*e^{-d_j\tau(y^*)}(1+k_1x^*)}{(1+k_1x^*+k_2y^*)^2},\\
\eta&=\frac{nbx^*y^*e^{-d_j\tau(y^*)}\tau'(y^*)(d-d_j)}{1+k_1x^*+k_2y^*}.
\end{align*}
  From \cite{22}, we know the distribution of the roods in an $4^{th}$-order polynomial
\begin{equation}\label{M10}
v^4+Q_1v^3+Q_2v^2+Q_3v+Q_4=0,
\end{equation}
which is:
\begin{lemma}
Considering the positive real root of equation (\ref{M10}), there are the following conclusions: \begin{itemize}
\item[(i)]  If $Q_4<0$, then equation (\ref{M10}) has at least one positive real root;
\item[(ii)]  If $Q_4\geq0$, and $\bigtriangleup\geq0$, then equation (\ref{M10}) have positive real roots iff $v_1>0$, and $h(v_1)<0$;
\item[(iii)]  If $Q_4\geq0$, and $\bigtriangleup<0$, then equation (\ref{M10}) have positive real roots iff exists $v^*\in\{v_1,v_2,v_3\}$ such that $v^*>0$, and $h(v^*)\leq0$ holds true.
\end{itemize}
\end{lemma}
where
\begin{align*}
h(v)&=v^4+Q_1v^3+Q_2v^2+Q_3v+Q_4,\\
M&=\frac{1}{2}Q_2-\frac{3}{16}Q_1^2,&   N &=\frac{1}{32}Q_1^3-\frac{1}{8}Q_1Q_2+Q_3,\\
\bigtriangleup&=(\frac{N}{2})^2+(\frac{M}{2})^3,&         \sigma &=\frac{-1+\sqrt{3}i}{2},\\
Y_1&=\sqrt[3]{-\frac{N}{2}+\sqrt{\bigtriangleup}}+\sqrt[3]{-\frac{N}{2}-\sqrt{\bigtriangleup}},&    Y_2&=\sigma\sqrt[3]{-\frac{N}{2}+\sqrt{\bigtriangleup}}+\sigma^2\sqrt[3]{-\frac{N}{2}-\sqrt{\bigtriangleup}}\\
Y_3&=\sigma^2\sqrt[3]{-\frac{N}{2}+\sqrt{\bigtriangleup}}+\sigma\sqrt[3]{-\frac{N}{2}-\sqrt{\bigtriangleup}},&   v_i&=Y_i-\frac{3Q_1}{4},~~~i=1,2,3.
\end{align*}
Consequently, we have the following:
\begin{theorem}
The positive equilibrium $(x^*,y^*)$ in system (\ref{N1}) is locally asymptotically stable provided that system (\ref{M5}) is permanent and
\begin{equation}\label{k2}
k_2>2\cdot \frac{nbe^{-d_j\tau(0)}-dk_1}{nre^{-d_j\tau(0)}},\ \ \ \ d\leq d_j
\end{equation}
holds true.
\end{theorem}
\begin{proof}
For the positive equilibrium $E^*(x^*,y^*)$, we have
\begin{equation}\label{PE}
r(1-\frac{x^*}{K})-\frac{by^*}{1+k_1x^*+k_2y^*}=0,~~~~~~~\frac{nbx^*e^{-d_j\tau(y^*)}}{1+k_1x^*+k_2y^*}-d=0.
\end{equation}
Thus
$$x^*=\frac{1}{2}(-\alpha +\sqrt{\alpha^2+4\beta}).$$
where
$$\alpha=\frac{K}{r}(\frac{nbe^{-d_j\tau(y^*)}-dk_1}{nk_2e^{-d_j\tau(y^*)}-r}),~~~~~~\beta=\frac{Kd}{nrk_2e^{-d_j\tau(y^*)}}.$$
nothing that $x^*<K$, then we have
$$x^*>\frac{1}{2}(-\alpha+|\alpha|)=-\alpha>K(1-\frac{nbe^{-d_j\tau(0)}-dk_1}{nrk_2e^{-d_j\tau(0)}})>K/2>0.$$
Since $d\leq d_j$, we have
$$A<0, \eta\leq0, B>0, C>0, D \leq d, r-\frac{2r}{K}x^*<0.$$
 The characteristic equation
\begin{equation}\label{M11}
(\lambda-A)(\lambda+d-\eta-De^{-d_j\tau(y^*)})+BCe^{-\lambda\tau(y^*)}=0.
\end{equation}

Thus the roots are given by the following equation:
\begin{equation}\label{M12}
G(\lambda)=\lambda^2+H_1\lambda+H_2+(N_1\lambda+N_2)e^{-\lambda\tau(y^*)}=0
\end{equation}
where
\begin{align*}
H_1&=d-\eta-A    &   H_2&=A\eta-Ad\\
N_1&=-D        &    N_2&=AD+BC\\
\end{align*}
Since
\begin{align*}
H_2+N_2& =A\eta-Ad+AD+EC\\
&=A(D-d+\eta)+BC\\
&>0
\end{align*}
Hence, zero is not the root of equation (\ref{M12}).

Now, let us prove that equation (\ref{M12}) has no purely imaginary roots.
Assume that equation (\ref{M12}) has a purely imaginary root $\lambda=iv$, where $v>0$. Substituting it into equation (\ref{M12}) and separating the real and the imaginary parts, we obtain
\begin{align*}
v^2-H_2&=N_1vsin(\tau(y^*v)+N_2cos(\tau(y^*v),\\
-H_1v&=N_1vcos(\tau(y^*v)-N_2sin(\tau(y^*v).
\end{align*}
Add the square of the above equation, we have
\begin{equation}\label{M13}
v^4+B_1v^2+B_2=0
\end{equation}
where $B_1=H_1^2-2H_2-N_1^2, B_2=H_2^2-N_2^2.$

Since the model assumes $\tau'(y^*)\geq0$, we will discuss the following two cases.
\noindent\textbf{Case 1}.  $\tau'(y^*)=0$,
we have $\eta=0$, then
\begin{equation*}
\begin{split}
B_1=&(d-A)^2+2Ad-D^2=A^2+(d+D)(d-D)>0,\\
B_2=&(H_2+N_2)(H_2-N_2)\\
=&(H_2+N_2)(-Ad-AD-BC)\\
=&(H_2+N_2)\Big[-A(\frac{nbx^*e^{-d_j\tau(y^*)}}{1+k_1x^*+k_2y^*}+\frac{nbx^*e^{-d_j\tau(y^*)}}{1+k_1x^*+k_2y^*}\frac{1+k_1x^*}{1+K_1x^*+k_2y^*})\\
&-\frac{nb^2x^*y^*e^{-d_j\tau(y^*)}(1+k_1x^*)(1+k_2y^*)}{(1+K_1x^*+k_2y^*)^4}\Big]\\
=&(H_2+N_2)\Big[-\frac{nbx^*e^{-d_j\tau(y^*)}}{1+k_1x^*+k_2y^*}(A+\frac{1+k_1x^*}{1+k_1x^*+k_2y^*}(r-\frac{2r}{K}x^*))\Big]\\
>&0.
\end{split}
\end{equation*}
Then the equation (\ref{M13}) has no positive roots, i.e. equation (\ref{M12}) has no purely imaginary roots.

\noindent\textbf{Case 2}.  $\tau'(y^*)>0$, we have $\eta\leq0$, then by Lemma 2, for the equation (\ref{M10}),
\begin{align*}
Q_1&=0, Q_2=B_1, Q_3=0,\\
Q_4& =B_2=(H_2+N_2)((H_2-N_2)\\
&=(H_2+N_2)(A\eta-Ad-AD-BC)\\
&>0.
\end{align*}
Therefore, equation (\ref{M13}) has  positive real roots can only be the case (ii) and (iii) in the Lemma $3$, however, $v_i=Y_i-3Q_1/4=Y_i=0,~i=1,2,3, $  the case (ii) and (iii) in the lemma $3$ cannot be established. Then the equation (\ref{M13}) has no positive roots, i.e., equation (\ref{M12}) has no purely imaginary roots, and each root of the characteristic equation has a negative real part. The proof is complete.
\end{proof}

\section{Global attractiveness}
In this section, we consider the global stability of $(x^*,y^*)$ in system (\ref{N1}). The next two lemmas are elementary and useful in the following discussion, which can be found in Gopalsamy \cite{GK} and Hirsh et al. \cite{HW}.
\begin{lemma}
($Barb\breve{a}lat$ Lemma). Let $a$ be a finite number and $f:  [a, \infty) \to R$ be a differentiable function. If $lim_{t\to\infty}f(t)$ exists (finite) and $f'$ is uniformly continuous on $[a, \infty)$, then $\lim_{t\to\infty} f'(t) = 0$.
\end{lemma}
\begin{lemma}
(Fluctuation Lemma). Let $a$ be a finite number and $f:  [a, \infty) \to R$ be a differentiable function. If $\liminf_{t\to\infty} f (t) < \limsup_{t\to\infty}f(t)$, then there exist sequences $\{t_n\} \uparrow\infty$ and $\{s_n\}\uparrow\infty$ such that $\lim_{n\to\infty} f(t_n) = \limsup_{t\to\infty} f(t)$, $f'(t_n)=0$ and $\lim_{n\to\infty}f(s_n)=\liminf_{t\to\infty}f(t)$, $f'(s_n)=0$.
\end{lemma}
Now, we are mainly interested in the global asymptotic stability of $(x^*,y^*)$. Before proceeding, we will need the following lemma.
\begin{lemma}
Let $v(t)$ be the solution of
\begin{equation}\label{N2}
v'(t)=\big[1-\tau'(v)v'(t)\big]\frac{a_1e^{-d_j\tau(v)}v(t-\tau(v))}{1+a_2v(t-\tau(v))}-a_3v(t)
\end{equation}
where the initial date $v(t)=\phi(t)\geq0$, $v(0)>0$, \mbox{for} $t\in [-\tau_M,0]$ and $a_i>0,\ \mbox{for}\  i=1,2,3$. Then $\lim_{t\to\infty}v(t)=\tilde{v}$ providing
\begin{equation*}
a_1e^{-d_j\tau(\tilde v)}>a_3, \ \mbox{where} \ \ \tilde{v}=\frac{a_1e^{-d_j\tau(\tilde v)}-a_3}{a_2a_3}.
\end{equation*}
\end{lemma}
\begin{proof}
First we prove the positiveness and eventually uniformly boundedness of $v(t)$.
Assume that, there exists $t^*>0$, such that $v(t^*)=0$ and $y(t)>0$ for every $t\in(0,t^*)$. Then we have
\begin{equation*}
v'(t)\geq-a_3v~~~ \forall t\in[-\tau_M,t^*]
\end{equation*}
Therefore
\begin{equation*}
v(t)\geq v_0e^{-a_3t}>0
\end{equation*}
let $t=t^*$, we have
\begin{equation*}
0=v(t^*)\geq v_0e^{-a_3t^*}>0
\end{equation*}
which is a contradiction. So no such $t^*$ exists, and we obtain the positiveness results.

Proof of boundedness is divided into two steps.
\begin{itemize}
\item[(i)]
Suppose that $v'(t)\geq0$ for all $t>T$ for some $T\geq0$. Then for $t>T+\tau_M$,
\begin{equation*}
\begin{split}
0\leq v'(t)&= \big[1-\tau'(v)v'(t)\big]\frac{a_1e^{-d_j\tau(v)}v(t-\tau(v))}{1+a_2v(t-\tau(v))}-a_3v(t)\\
&\leq\frac{a_1e^{-d_j\tau(v)}v(t)}{1+a_2v(t)}-a_3v(t)\\
&=v(t)\Big[\frac{a_1e^{-d_j\tau(v)}}{1+a_2v(t)}-a_3\Big]
\end{split}
\end{equation*}
since $v(t-\tau(v))\leq v(t)$.  This means that
$$v(t)\leq \frac{a_1e^{-d_j\tau(v)}-a_3}{a_2a_3}\leq \frac{a_1e^{-d_j\tau(0)}-a_3}{a_2a_3}, \ \ \mbox{for} \ \ t>T+\tau_M.$$
\item[(ii)]
Assume that $v(t)$ is not eventually monotonic. There is a sequence $\{t_n\}_{n=1}^{\infty}$ such that $v'(t_n)=0$, and $v(t_n)$ is a local maximum. We can further choose the subsequence (still
denote $\{t_n\}_{n=1}^{\infty})$ such that $v(t)\leq v(t_n)$ for all $0<t<t_n$ and $n$. Then by a similar analysis at $t=t_n$, it follows that the solutions $v(t)$ is bounded above by a bound.
\end{itemize}
Let us next deal with the case when $v(t)$ is eventually monotonic. For this case, there exists $0\leq v<\infty$  such that $\lim_{t\to\infty} v(t)=\tilde{v}$ and $\lim_{t\to\infty}v '(t)=0$. Hence from system (\ref{N2}), taking the limit as $t\to\infty$, we get that
\begin{equation*}
0=\lim_{t\to\infty}v'(t)=\tilde{v}\Big(\frac{a_1e^{-d_j\tau(\tilde v)}}{1+a_2\tilde{v}}-a_3\big)
\end{equation*}
Thus $\tilde{v}=0$ or $\tilde{v}=\frac{a_1e^{-d_j\tau(\tilde v)}-a_3}{a_2a_3}$ if $a_1e^{-d_j\tau(\tilde v)}>a_3$. So,  this limit must be an equilibrium of (\ref{N2}) and is therefore either zero or the value stated. Zero is ruled
out since a standard linearized analysis yields that the zero solution of (\ref{N2}) is linearly unstable under the stated condition on $a_1e^{-d_j\tau(\tilde v)}>a_3$. Therefore, $\tilde{v}=\frac{a_1e^{-d_j\tau(\tilde v)}-a_3}{a_2a_3}$.

The rest of the case to discuss is that in which $v(t)$ is neither eventually monotonically
increasing nor decreasing. Now, we assume that $v(t)$ is oscillatory. Then $v(t)$ has an infinite sequence of local maxima and define the sequence $\{t_j\}$ as those times for which $v'(t_j)=0$ and $v''(t_j)<0$. Here, we will only discuss in detail the case of the local maximum $v(t_j)>\tilde{v}$ for all $j = 1, 2, 3,\cdots$, and other cases can be dealt with analogously.

Now, we prove that $\sup_{t\geq t_1}v(t)=v(t_k)$ for some integer $k$. Otherwise, after every local maximum $v(t_j)$ there is another that is higher, and therefore a subsequence of $\{t_j\}$ (still relabelled $\{t_j\}$) can be chosen with the property that $v(t)<v(t_j)$ for all $t_1\leq t <t_j$ and each $j$. The subsequence is selected by including each local maximum which is smaller than every one before it. By assumption $\tau(v)>0$ and $t_j-\tau(v(t_j))<t_j$, for each $j$
\begin{equation*}
\begin{split}
0=v'(t_j)&=\big[1-\tau'(v(t_j))v'(t_j)\big]\frac{a_1e^{-d_j\tau(v(t_j))}v\big(t_j-\tau(v(t_j))\big)}{1+a_2v(t_j-\tau(v(t_j)))}-a_3v(t_j)\\
&<v(t_j)\Big[\frac{a_1e^{-d_j\tau(\tilde v)}}{1+a_2v(t_j)}-a_3\Big]\\
&<v(t_j)\Big[\frac{a_1e^{-d_j\tau(\tilde v)}}{1+a_2\tilde{v}}-a_3\Big]=0
\end{split}
\end{equation*}
this is a contradiction. So, $\sup_{t\geq t_1}v(t)=v(t_k)$ for some integer $k$ and we let $s_1 =t_k$. Now, by applying the same analysis to the interval $t\geq t_{k+1}$, the existence of a $t_l(l>k)$ with $\sup_{t\geq t_{k+1}}v(t)=v(t_l)$ can be obtained, and we set $s_2=t_l$. Continuing this process, we obtain an infinite sequence $\{s_j\}$ of times such that $s_{j+1}>s_j$ as $s_j\to\infty$, and $v(t)\geq v(s_j)$ for all $t>s_j$, and $v'(s_j)=0$.

Let $y(t)=v(t)-\tilde{v}$. Next, we will prove that $y(t)\to0$ as $t\to\infty$. We have got a sequence $y(s_j)>y(s_{j+l})>0$ (since $v(s_j)\geq v(s_{j+l})$ and $v(s_j)>\tilde{v}$), and it is now enough to show that $y(s_j)\to0$ as $j\to\infty$. In terms of $y$, equation (\ref{N2}) becomes, at $t=s_j$,
\begin{equation*}
\begin{split}
0=y'(s_j)&=\big[1-\tau'(v(s_j))v'(s_j)\big]\frac{a_1e^{-d_j\tau(v(s_j))}v[s_j-\tau(v(s_j))]}{1+a_2v[s_j-\tau(v(s_j))]}-a_3v(s_j)\\
&=\frac{\Big[a_1[y(s_j-\tau(y(s_j)+\tilde{v})]+a_1\tilde{v}\Big]e^{-d_j\tau(y(s_j)+\tilde {v})}}{1-a_2y[s_j-\tau(y(s_j)+\tilde{v})]+a_2\tilde{v}}-a_3(y(s_j)+\tilde{v}).\\
&<\frac{\Big[a_1[y(s_j-\tau(y(s_j)+\tilde{v})]+a_1\tilde{v}\Big]e^{-d_j\tau(\tilde {v})}}{1-a_2y[s_j-\tau(y(s_j)+\tilde{v})]+a_2\tilde{v}}-a_3(y(s_j)+\tilde{v}).
\end{split}
\end{equation*}
since $\tau(\tilde v)<\tau\big(y(s_j)+\tilde{v}\big).$
By the sequence $\{s_j\}$, we choose a final subsequence, once again denoted $\{s_j\}$, so that $s_j-\tau_M\geq s_{j-1}$. Then $y(s_j-s)<y(s_{j-1})$ for all $s\in [0, \tau_M]$ and therefore
\begin{equation*}
\begin{split}
a_3(y(s_j)+\tilde{v})&<\frac{\Big[a_1y[s_j-\tau(y(s_j)+\tilde{v})]+a_1\tilde{v}\Big]e^{-d_j\tau(\tilde {v})}}{1+a_2y[s_j-\tau(y(s_j)+\tilde{v})]+a_2\tilde{v}}\\
&<\frac{[a_1y(s_{j-1})+a_1\tilde{v}]e^{-d_j\tau(\tilde {v})}}{1+a_2y(s_{j-1})+a_2\tilde{v}}\\
&<\frac{[a_1y(s_{j-1})+a_1\tilde{v}]e^{-d_j\tau(\tilde {v})}}{1+a_2\tilde{v}}
\end{split}
\end{equation*}
so that
\begin{equation*}
y(s_j)<\frac{a_1e^{-d_j\tau(\tilde {v})}}{a_3(1+a_2\tilde{v})}y(s_{j-1})
\end{equation*}
noting that $\tilde{v}=\frac{a_1e^{-d_j\tau(\tilde {v})}-a_3}{a_2a_3}$, we have $\frac{a_1e^{-d_j\tau(\tilde {v})}}{a_3(1+a_2\tilde{v})}=1$ and it is independent of $j$. Therefore, $y(s_j)\to 0$ as $j\to\infty$. We summarize that $\lim_{t\to\infty}v(t) =\tilde{v}$ and complete the proof of this theorem.
\end{proof}
\begin{remark}
The global stability of a single-population model without delay time-derivative  term is discussed in \cite{16}. Now we use the same method to prove that for a single-population model with delay time-derivative  term, the conclusion is still true.
\end{remark}

Next, we prove the global stability of $(x^*,y^*)$ in system (\ref{N1}).
\begin{theorem}
The positive equilibrium $(x^*,y^*)$ in system (\ref{N1}) is globally attractive provided that $\frac{nbe^{-d_j\tau(y^*)}K}{1+k_1K}>d$ and
\begin{equation}\label{K2}
k_2>\max\Big\{ \frac{bK(nbe^{-d_j\tau(y^*)}-dk_1)}{r\big[(nbe^{-d_j\tau(y^*)}-dk_1)K-d\big]}, \frac{bK(nbe^{-d_j\tau(y^*)}-dk_1)}{rd}, \frac{b}{r}\Big\}
\end{equation}
holds true.
\end{theorem}
\begin{proof}
By the first equation of (\ref{N1}) and the arguments to Theorem $2$, for sufficiently small $\varepsilon>0$, there is a $T_1>0$ such that $x(t)<K+\varepsilon=\overline{x}_1$ for $t\geq T_1$. It is easy to see $x^*<\overline{x}_1$ for $t\geq T_1$. Replacing this inequality into the second equation of (\ref{N1}), since $1-\tau'(y)y'(t)>0$, we have
$$y'(t)<\big[1-\tau'(y)y'(t)\big]\frac{nbe^{-d_j\tau(y)}\overline{x}_1y(t-\tau(y))}{1+k_1\overline{x}_1+k_2y(t-\tau(y))}-dy(t),~~t\geq T_1+\tau_M$$
Consider the system
\begin{equation*}
\left\{
             \begin{split}
             v'(t)=&\big[1-\tau'(y)y'(t)\big]\frac{nbe^{-d_j\tau(y)}\overline{x}_1v(t-\tau(v))}{1+k_1\overline{x}_1+k_2v(t-\tau(v))}-dv(t), ~~t\geq T_1+\tau_M  \\
             v(t)=&y(t), ~~ t\in[t_1,T_1+\tau_M]
             \end{split}
\right.
\end{equation*}
Noting $nbe^{-d_j\tau(y^*)}\overline{x}_1-d(1+k_1\overline{x}_1)>nbe^{-d_j\tau(y^*)}K-d(1+k_1K)>0$. Thus by Lemma $6$, we have
$$\lim_{t\to\infty}v(t)=\frac{nbe^{-d_j\tau(y^*)}\overline{x}_1-d(1+k_1\overline{x}_1)}{k_2d}>0.$$
By the Lemma $1$, we have $y(t)\leq v(t)$, $t\geq T_1+\tau_M$. Then for the sufficiently small $\varepsilon>0$, there exists $T_2>T_1+\tau_M$ such that
\begin{equation}\label{Q1}
y(t)<\frac{nbe^{-d_j\tau(y^*)}\overline{x}_1-d(1+k_1\overline{x}_1)}{k_2d}+\varepsilon =\overline{y}_1,~~t\geq T_2.
\end{equation}
By the first equation of (\ref{PE}), we have $$y^*<\frac{nbe^{-d_j\tau(y^*)}\overline{x}_1-d(1+k_1\overline{x}_1)}{k_2d}<\overline{y}_1, ~~t\geq T_2.$$

Replacing (\ref{Q1}) into the first equation of (\ref{N1}), we have
$$x'(t)>rx(t)(1-\frac{x(t)}{K})-\frac{bx(t)\overline{y}_1}{1+k_2\overline{y}_1},~~t\geq T_2.$$
By (\ref{K2}), $r>\frac{b}{k_2}>\frac{b\overline{y}_1}{1+k_2\overline{y}_1}$. Using the comparison theorem, for sufficiently small $\varepsilon>0$, there is a $T_3>T_2$ such that
\begin{equation}\label{Q2}
x(t)>z^*-\varepsilon=\underline{x}_1>0,~~t\geq T_3.
\end{equation}
where $z^*=K\big[1-\frac{b\overline{y}_1}{r(1+k_2\overline{y}_1)}\big]>0$ is the positive root for the equation
$$rx(t)(1-\frac{x(t)}{K})-\frac{bx(t)\overline{y}_1}{1+k_2\overline{y}_1}=0.$$
By the second equation of (\ref{PE}), we have
$$r(1-\frac{x^*}{K})<\frac{b\overline{y}_1}{1+k_2\overline{y}_1}=r(1-\frac{z^*}{K}), ~~t\geq T_3.$$
i.e.,for sufficiently small $\varepsilon>0$, $$x^*>z^*-\varepsilon=\underline{x}_1, ~~t\geq T_3.$$

Replacing (\ref{Q2}) into the second equation of (\ref{N1}), we have
$$y'(t)>\big[1-\tau'(y)y'(t)\big]\frac{nbe^{-d_j\tau(y)}\underline{x}_1y(t-\tau(y))}{1+k_1\underline{x}_1+k_2y(t-\tau(y))}-dy(t),~~t\geq T_3+\tau_M$$
By (\ref{Q2}), we have
\begin{equation*}
\begin{split}
nbe^{-d_j\tau(y^*)}\underline{x}_1-d(1+k_1\underline{x}_1)=&(nbe^{-d_j\tau(y^*)}-dk_1)\big\{K\big[1-\frac{b\overline{y}-1}{r(1+k_2\overline{y}_1)}\big]-\varepsilon\big\}-d\\
>&(nbe^{-d_j\tau(y^*)}-dk_1)\big\{K\big[1-\frac{b}{rk_2}\big]-\varepsilon\big\}-d\\
=&\frac{(nbe^{-d_j\tau(y^*)}-dk_1)(K-\varepsilon)-d}{k_2}\\
&\cdot\big\{k_2-\frac{bK(nbe^{-d_j\tau(y^*)}-dk_1)}{r[(nbe^{-d_j\tau(y^*)}-dk_1)(K-\varepsilon)-d]}\big\}.
\end{split}
\end{equation*}
Using (\ref{K2}), we can get
\begin{equation}\label{Z1}
nbe^{-d_j\tau(y^*)}\underline{x}_1-d(1+k_1\underline{x}_1)>0\ \ \ \mbox{for sufficiently small} \ \ \varepsilon.
\end{equation}
By Lemma $6$ and the similar arguments to $\overline{y}_1$, for the above selected $\varepsilon>0$, there exists $T_4>T_3+\tau$ such that
\begin{equation}\label{Q3}
y(t)>\frac{nbe^{-d_j\tau(y^*)}\underline{x}_1-d(1+k_1\underline{x}_1)}{k_2d}-\varepsilon =\underline{y}_1,~~t\geq T_4.
\end{equation}
and $$y^*>\underline{y}_1, ~~t\geq T_4.$$
Therefore we have that
$$\underline {x}_1<x(t)<\overline{x}_1,\ \ \underline{y}_1<y(t)<\overline{y}_1,\ \ t\geq T_4,$$
$$\underline {x}_1<x^*<\overline{x}_1,\ \ \underline{y}_1<y^*<\overline{y}_1,\ \ t\geq T_4,$$

hold for system (\ref{N1}).

Replacing (\ref{Q3}) into the first equation of (\ref{N1}), we have
$$x'(t)<rx(t)(1-\frac{x(t)}{K})-\frac{bx(t)\underline{y}_1}{1+k_1\overline{x}_1+k_2\underline{y}_1},~~t\geq T_4.$$
Since $r-\frac{b\underline{y}_1}{1+k_1\overline{x}_1+k_2\underline{y}_1}>r-\frac{b\overline{y}_1}{1+k_2\overline{y}_1}>0$, by the comparison theorem, for sufficiently small $\varepsilon>0$, there is a $T_5>T_4$ such that
\begin{equation}\label{Q4}
x(t)<z_1^*+\varepsilon=\overline{x}_2>0,~~t\geq T_5,
\end{equation}
with $z_1^*=K\big[1-\frac{b\underline{y}_1}{r(1+k_2\underline{y}_1)}\big]>0$. By the similar arguments to $\underline{x}_1$, we have $$x^*< \overline{x}_2, t\geq T_5.$$
From the definition of $\overline{x}_2$ we get
$$\overline{x}_2<K<\overline{x}_1.$$

Replacing (\ref{Q4}) into the second equation of (\ref{N1}), we have
$$y'(t)<\big[1-\tau'(y)y'(t)\big]\frac{nbe^{-d_j\tau(y)}\overline{x}_2y(t-\tau(y))}{1+k_1\overline{x}_2+k_2y(t-\tau(y))}-dy(t),~~t\geq T_5+\tau_M$$
Since $\overline{x}_2>\underline{x}_1$ and noting (\ref{Z1}), we have
$nbe^{-d_j\tau(y^*)}\overline{x}_2-d(1+k_1\overline{x}_2)>nbe^{-d_j\tau(y^*)}\underline{x}_1-d(1+k_1\underline{x}_1)>0$.
Thus using arguments similar to above, for the sufficiently small $\varepsilon>0$, there is a $T_6>T_5+\tau_M$ such that
\begin{equation}\label{Q5}
y(t)<\frac{nbe^{-d_j\tau(y^*)}\overline{x}_2-d(1+k_1\overline{x}_2)}{k_2d}+\varepsilon =\overline{y}_2,~~t\geq T_6.
\end{equation}
by (\ref{Q1}), (\ref{Q5}) we get $y^*<\overline{y}_2<\overline{y}_1.$

Replacing (\ref{Q5}) into the first equation of (\ref{N1}), we have
$$x'(t)>rx(t)(1-\frac{x(t)}{K})-\frac{bx(t)\overline{y}_2}{1+k_2\overline{y}_2},~~t\geq T_6.$$
From (\ref{K2}), $r>\frac{b}{k_2}>\frac{b\overline{y}_1}{1+k_2\overline{y}_1}>\frac{b\overline{y}_2}{1+k_2\overline{y}_2}$. Then by the comparison theorem, for sufficiently small $\varepsilon>0$, there is a $T_7>T_6$ such that
\begin{equation}\label{Q6}
x(t)>z_2^*-\varepsilon=\underline{x}_2>0,~~t\geq T_7,
\end{equation}
where $z_2^*=K\big[1-\frac{b\overline{y}_2}{r(1+k_2\overline{y}_2)}\big]>0$. By the definition of $\underline{x}_2$, we have $\underline{x}_2>\underline{x}_1>x^*$.

Replacing (\ref{Q6}) into the second equation of (\ref{N1}), then by arguments similar to those for $\overline{y}_2$, we get that there exists a $T_8>T_7+\tau_M$ such that
\begin{equation}\label{Q7}
y(t)>\frac{nbe^{-d_j\tau(y)}\underline{x}_2-d(1+k_1\underline{x}_2)}{k_2d}-\varepsilon =\underline{y}_2,~~t\geq T_8.
\end{equation}
and we get $y^*>\underline{y}_2>\underline{y}_1$.

Therefore we have that
$$0<\underline {x}_1<\underline {x}_2<x(t)<\overline{x}_2<\overline{x}_1,\ \ 0<\underline{y}_1<\underline{y}_2<y(t)<\overline{y}_2<\overline{y}_1,\ \ t\geq T_8.$$
$$0<\underline {x}_1<\underline {x}_2<x^*<\overline{x}_2<\overline{x}_1,\ \ 0<\underline{y}_1<\underline{y}_2<y^*<\overline{y}_2<\overline{y}_1,\ \ t\geq T_8.$$
Repeating the above arguments, we get the four sequences
$\{\overline{x}_n\}_{n=1}^\infty$,
$\{\underline{x}_n\}_{n=1}^\infty$,
$\{\overline{y}_n\}_{n=1}^\infty$,
$\{\underline{y}_n\}_{n=1}^\infty$ with
\begin{equation}\label{Q8}
\begin{split}
0<&\underline{x}_1<\underline{x}_2<\cdots<\underline{x}_n<x(t)
  <\overline{x}_n<\cdots<\overline{x}_2<\overline{x}_1,
\\
0<&\underline{y}_1<\underline{y}_2<\cdots<\underline{y}_n<y(t)
<\overline{y}_n<\cdots<\overline{y}_2<\overline{y}_1,\ \ t\geq
T_{4n}.
\end{split}
\end{equation}
\begin{equation}\label{Zqq}
\begin{split}
0<&\underline{x}_1<\underline{x}_2<\cdots<\underline{x}_n<x^*
  <\overline{x}_n<\cdots<\overline{x}_2<\overline{x}_1,
\\
0<&\underline{y}_1<\underline{y}_2<\cdots<\underline{y}_n<y^*
<\overline{y}_n<\cdots<\overline{y}_2<\overline{y}_1,\ \ t\geq
T_{4n}.
\end{split}
\end{equation}
From (\ref{Q8}) follows that the limit of each sequence in $\{\overline{x}_n\}_{n=1}^\infty$,
$\{\underline{x}_n\}_{n=1}^\infty$,
$\{\overline{y}_n\}_{n=1}^\infty$,
$\{\underline{y}_n\}_{n=1}^\infty$ exist. Denote
$$\overline{x}=\lim_{n\rightarrow\infty}\overline{x}_n,\ \
\overline{y}=\lim_{n\rightarrow\infty}\overline{y}_n,\ \
\underline{x}=\lim_{n\rightarrow\infty}\underline{x}_n,\ \
\underline{y}=\lim_{n\rightarrow\infty}\underline{y}_n.$$
thus we get $\overline{x}\geq \underline{x},\ \overline{y}\geq
\underline{y}.$ To complete the proof, it suffices to prove
$\overline{x}=\underline{x}=x^*,\ \overline{y}=\underline{y}=y^*.$

By the definition of $\overline{y}_n,\ \underline{y}_m,$ we have
$$\overline{y}_n=\frac{nbe^{-d_j\tau(0)}\overline{x}_n-d(1+k_1\overline{x}_n)}{k_2d}+\varepsilon,\ \
\underline{y}_m=\frac{nbe^{-d_j\tau(0)}\underline{x}_m-d(1+k_1\underline{x}_m)}{k_2d}-\varepsilon.$$
then we get
\begin{equation}\label{Q9}
\overline{y}_n-\underline{y}_m=\frac{nbe^{-d_j\tau(0)}-dk_1}{k_2d}(\overline{x}_n-\underline{x}_m)+2\varepsilon.
\end{equation}
By the definition of $\overline{x}_n,\ \underline{x}_n$ and using (\ref{Q9}), we have
\begin{equation*}
\begin{split}
\overline{x}_n-\underline{x}_n&=K\Big[1-\frac{b\underline{y}_{n-1}}{r(1+k_2\underline{y}_{n-1})}\Big]-K\Big[1-\frac{b\underline{y}_n}{r(1+k_2\underline{y}_n)}\Big]+2\varepsilon\\
&=\frac{bK}{r}\Big[\frac{\overline{y}_n-\underline{y}_{n-1}}{(1+k_2\underline{y}_{n-1})(1+k_2\underline{y}_n)}\Big]+2\varepsilon\\
&=\frac{bK}{r}\frac{[nbe^{-d_j\tau(0)}-dk_1]/k_2d(\overline{x}_n-\underline{x}_{n-1})+2\varepsilon}{(1+k_2\underline{y}_{n-1})(1+k_2\underline{y}_n)}+2\varepsilon\\
&<\frac{bK}{k_2dr}[nbe^{-d_j\tau(0)}-dk_1](\overline{x}_n-\underline{x}_{n-1})+2\varepsilon\big(1+\frac{bK}{r}\big).
\end{split}
\end{equation*}
Let $n\to\infty$, then we have
$$\overline{x}-\underline{x}\leq\frac{bK}{k_2dr}[nbe^{-d_j\tau(0)}-dk_1](\overline{x}-\underline{x})+2\varepsilon\big(1+\frac{bK}{r}\big),$$
thus$$\Big\{1-\frac{bK}{k_2dr}[nbe^{-d_j\tau(0)}-dk_1]\Big\}(\overline{x}-\underline{x})\leq2\varepsilon\big(1+\frac{bK}{r}\big).$$
From (\ref{K2}), we have $1-\frac{bK}{k_2dr}[nbe^{-d_j\tau(0)}-dk_1]>0$, and noting that $\varepsilon>0$ can be arbitrarily small, then we have
$\overline{x}=\underline{x}.$ By (\ref{Q9}) and let
$n,m\rightarrow\infty,$ we get $\overline{y}=\underline{y}.$
From (\ref{Zqq}), $\overline{x}=\underline{x}=x^*,\ \overline{y}=\underline{y}=y^*.$ This proves Theorem $8$. \end{proof}
\begin{corollary}\label{coro1}
The positive equilibrium $(x^*,y^*)$ in system (\ref{N1}) is globally asymptotically stable provided that $\frac{nbe^{-d_j\tau(y^*)}K}{1+k_1K}>d$  and that the conditions (\ref{k2}), (\ref{K2}) hold true.
\end{corollary}

\begin{remark}
 From Theorem and Corollary \ref{coro1}, it is shown that if the system (\ref{N1})  is permanent, that the   derivative of SDTD on the state $y$ is small enough  such that $|\tau(y^*)-\tau(0)|$ is small enough, and that the predator interference $k_2$ is large enough, then  the coexistence equilibrium $(x^*,y^*)$ in system (\ref{N1})   is globally asymptotically stable. Theorem 8 and Corollary \ref{coro1} directly extend  \cite[Theorem 4.1]{6}
\end{remark}

\section{Conclusions and discussions}
In this paper, based on the biological observations that during World War II the maturation time of seals and whales was not a fixed value, but depended on the mature population, starting with an age-structured model (\ref{M1}), we formulated and analyzed a prey-predator stage-structured model with SDTD.

Compared with the previous SDTD models \cite{10,13,14}, model (\ref{M2}) is not directly changing the constant delay $\tau$ into a SDTD $\tau(y(t))$ but was obtained by reducing the age-structured population model, which has the delay time-derivative  term $\tau '(y(t))\cdot y'(t)\leq 1$. Biologically speaking, model (\ref{M2}) is appropriate in terms of population modeling. On the one hand, with the SDTD, the changes in the number of mature individuals depend on reproduction and death and the changing definition of maturity, which is in line with the delay time-derivative  term $ \tau '(y(t))\cdot y'(t)\leq 1$. On the other hand, we can represent $y(t)$ and $y_j(t)$ in an integral form by some biological inductions, namely
\begin{equation*}
\begin{split}
 y(t)&=e^{-dt} \int_{0}^{t} e^{ds}n(1-\tau'(y(s))y'(s)) e^{-d_j\tau(y(s))}f(x(s-\tau(y(s)),y(s-\tau(y(s)))\\
&y(s-\tau(y(s)))ds\\
y_j(t)&=\int_{t-\tau(y)}^{t}nf(x(s),y(s))y(s)e^{d_j(t-s)}ds
\end{split}
\end{equation*}
Taking the derivatives of $y(t)$ and $y_j(t)$, we obtain the second and third equation of model (\ref{M2}).

 From a biological point of view, we show that $t-\tau(y(t))$ should be a strictly increasing function of $t$ without any conditions and the derivative with respect to time of the SDTD $\tau(y(t))$ is strictly less than one. In addition, it is biologically reasonable for the assumption of the delay $\tau(y(t))$. The biological phenomenon mentioned above, a non-decreasing delay, implies that a more mature predator leads to a longer developmental duration, and it makes clear the stabilizing effect \cite{5,23}.

 Mathematically compared with \cite{13,14}, first of all, the positivity and boundedness of
solutions are discussed, which do not need the stringent condition on $\tau(y(t))$ to ensure the positivity of $y$ and $y_j$.  Then we give the conditions which are both necessary and sufficient for the permanence and extinction of system (\ref{M2}), and local asymptotic stability of trivial and the boundary equilibria is investigated. Finally, taking the BD-type functional response function as an example, the local asymptotic stability and global attractiveness of positive equilibrium of the model is discussed. It shows that as long as $k_2$ is big enough, the positive equilibrium $E(x^*,y^*,y_j^*)$ will reach a steady state.

Further research in this direction may consider more realistic complex models, for example,
\begin{equation*}
\begin{split}
x'(t)=&rx(t)(1-\frac{x(t)}{K})-f(x(t),y(t))y(t),\\
y'(t)=&n(1-\tau'(x)x'(t)) e^{-d_j\tau(x)}f(x(t-\tau(x),y(t-\tau(x)))y(t-\tau(x))-dy(t),\\
y_j'(t)=&-n(1-\tau'(x)x'(t)) e^{-d_j\tau(x)}f(x(t-\tau(x),y(t-\tau(x)))y(t-\tau(x))\\
&+nf(x(t),y(t))y(t)-d_jy_j(t).
\end{split}
\end{equation*}
where $\tau (x)$ is a decreasing continuously  differentiable bounded function of the prey population $x(t)$. And,
\begin{equation*}
\begin{split}
x'(t)=&rx(t)(1-\frac{x(t)}{K})-f(x(t),y(t))y(t),\\
y'(t)=&n(1-\tau'(z)z'(t)) e^{-d_j\tau(z)}f(x(t-\tau(z),y(t-\tau(z)))y(t-\tau(z))-dy(t),\\
y_j'(t)=&-n(1-\tau'(z)z'(t)) e^{-d_j\tau(z)}f(x(t-\tau(z),y(t-\tau(z)))y(t-\tau(z))\\
&+nf(x(t),y(t))y(t)-d_jy_j(t).
\end{split}
\end{equation*}
where $z(t)$ is the total number of mature and immature predators $y(t)+y_j(t)$, and $\tau (z)$ is an increasing continuously  differentiable bounded function of the predator population $y(t)+y_j(t)$, which will be left as our future work.

\section*{Acknowledgments}
The authors would like to thank Dr Xianning Liu  for his valuable discussions.
Q. Zhang and S. Liu are supported by  the National Natural
Science Foundation of China (No.11871179 and No.11771374).  Y. Yuan is supported by ......     Y. Lv is supported by  the National Natural
Science Foundation of China (No.11871371).

\end{document}